\documentclass[12pt]{article}
\usepackage{graphics,amsmath,latexsym,amssymb, graphicx, epstopdf, mathrsfs,epsfig, amsthm}
\usepackage{mathtools,authblk}
\usepackage{setspace}
\usepackage{color}
\usepackage{enumerate}
\usepackage{enumitem}
\newcommand{\N}{\ensuremath{\mathbb{N}}}

\def\T{{\cal{T}}}

\def\S{{\cal{S}}}

\def\B{{\cal {B}}}

\def\max{{\rm max }}

\def\mod{{\rm mod\; }}

\def\GDD{{\sf{GDD}}}

{\hfill\qed\end{trivlist}}
\newtheorem{Theorem}{Theorem}[section]

\newtheorem{Lemma}[Theorem]{Lemma}
\newtheorem{rem}[Theorem]{Remark}

\newtheorem{define}[Theorem]{Definition}

\newtheorem{exer}[Theorem]{Exercise}

\newtheorem{ex}[Theorem]{Example}

\begin{document}
\setcounter{footnote}{1}

\title{Group Divisible Designs with $\lambda_1=3$ and Large Second Index}

\medskip

\author{
Chariya Uiyyasathian\footnote{Corresponding Author} and Nataphan Kitisin\\
Department of Mathematics and Computer Science,\\ Faculty of Science, Chulalongkorn University,\\ Phyathai Rd., Patumwan,
Bangkok, 10330, THAILAND \\ {\tt chariya.u@chula.ac.th, nataphan.k@chula.ac.th}
}

\date{}

\maketitle

\begin{abstract}

 A  group divisible design \GDD$(m,n;\lambda_1,\lambda_2)$, is an ordered pair $(V, \B)$ where $V$ is an
$(m+n)$-set of symbols while $\B$ is a collection of $3$-subsets (called
blocks) of  $V$  satisfying the following properties: the $(m+n)$-set
is divided into 2 groups of size $m$ and of size $n$: each
pair of symbols from the same group occurs  in exactly $\lambda_1$
blocks in $\B$: and each pair of symbols from different groups occurs
in exactly $\lambda_2$ blocks in $\B$. $\lambda_1$ and $\lambda_2$ are referred to as first index and second index, respectively.
 Here, we focus on an existence problem of \GDD s when
$\lambda_1=3$ and $\lambda_2>3$. We obtain the necessary conditions and prove that
these conditions are sufficient for most of the cases.\\

\end{abstract}



\section{Introduction}

A  {\it group divisible design} $\GDD(v=v_1+v_2+\dots +v_g, g, k ;
\lambda_1, \lambda_2)$ is an ordered pair $(V, \B)$ where $V$ is a
$v$-set of symbols while $\B$ is a collection of $k$-subsets (called
\textit{blocks}) of  $V$  satisfying the following properties: the
$v$-set is partitioned into $g$ groups of sizes $v_1, v_2, \dots,
v_g$; each pair of symbols from the same group occurs  in exactly
$\lambda_1$ blocks in $\B$; and each pair of symbols from
different groups occurs in exactly $\lambda_2$ blocks in $\B$.

The problem of the existence of group divisible design has been
interested for quite a long time. In 1952, Bose and Shimamoto published
a work on the classification of certain designs \cite{BO}. The case when $g=2$
and $k=3$ is of such highly interest recently. From now on, we will
use the notation \GDD$(m,n;\lambda_1,\lambda_2)$ to represent
\GDD$(v=m+n,2,3;\lambda_1,\lambda_2)$).

Group divisible designs can be described graphically as follows.
Let $\lambda K_v$ denote the graph on $v$ vertices in which each
pair of distinct vertices is joined by $\lambda$ edges. Let $G_1$
and $G_2$ be vertex disjoint graphs. The graph $G_1 \vee_\lambda
G_2$ is formed from the union of $G_1$ and $G_2$ by joining each
vertex in $G_1$ to each vertex in $G_2$ with $\lambda$ edges.
Given a subgraph $G$ of a graph $H$, a {\it
$G$-decomposition} of a graph $H$ is a partition of the edge set
of $H$ such that each element of the partition induces a copy of
$G$. Thus the existence of a $\GDD(m, n; \lambda_1, \lambda_2)$ is
easily seen to be equivalent to the existence of a
$K_3$-decomposition of ${\lambda_1} K_m \vee_{\lambda_2}
{{\lambda_1}K_n}$. In particular, the case where $\lambda_1=\lambda_2=\lambda$
is equivalent to $K_3$-decomposition of $\lambda K_{m+n}$. Such result is known as $\lambda$-fold triple systems and appears in many
standard textbooks (see \cite{LR}). Series of the research articles
had been devoted to solve the problem of existence of group
divisible designs with certain parameters \GDD$(m,n;\lambda_1,\lambda_2)$
where $\lambda_1>\lambda_2$ (eg.
\cite{1L},\cite{D1},\cite{nittiya},\cite{L23}). Although the case where
$\lambda_1 <\lambda_2$ is considered more difficult, the progress has been made. The case
where $\lambda_1=1, \lambda_2=2$ and $m<\frac{n}{2}$ was
established in 2012 \cite{12}. Later on, $\lambda_1=1, \lambda_2=3$,
the problem was partially solved in 2015 \cite{13}. For $\lambda_1=2$, the problem has been solved recently for most of the cases wherever $\lambda_2>2$ \cite{14}. Furthermore, when a
\GDD$(m,n;\lambda_1,\lambda_2)$ is {\em gregarious} (each
block in the design contains elements from both groups), El-Zanati
et al. found all $(m,n)$ for such gregarious \GDD$(m,n;1,2)$ to exist
\cite{12G}. Up to date, no other result where
$\lambda_1<\lambda_2$ has been known.

In this paper we continue along this line of work. In particular,
we solve the existence problem of a $\GDD(m, n; 3, \lambda)$, where
$\lambda > 3$. First we note that The existence of a $K_3$-decomposition of ${\lambda_1} K_m
\vee_{\lambda_2} {{\lambda_1}K_n}$ yields the necessary
conditions of the existence of our designs.

For sufficiency, we will provide the constructions using graph
decompositions. One of the major tools employed here
is the result from the Alspach's problem \cite{AL}. In 1981, Alspach asked
whether it is possible to decompose the complete graph on $n$
vertices, $K_n$, into $t$ cycles of specified lengths
$m_1, . . . , m_t$ given that the obvious necessary conditions are
satisfied. It turns out that a decomposition of $K_n$ into $3$-cycles is equivalent to a Steiner triple system of order $n$ (see \cite{LR}).
Also, there have been many papers on the case where the lengths of the cycles in the decomposition may vary.
Balister \cite{BA} has verified by a computer that Alspach's problem holds for $n\leq 14$. After series of research articles, Bryant, Horsley and
Pettersen finally solved the problem in \cite{HC} as the result of
the following theorem.

\begin{Theorem}\label{HC}\textnormal{\cite{HC}}
\begin{enumerate}
\item There is a decomposition ${G_1,G_2,...,G_t}$ of $K_n$ in
which $G_i$ is an $m_i$-cycle for $i=1,2,...,t$ if and only if $n$
is odd, $3\leq m_i\leq n$ for $i=1,2,...,t$ and
$m_1+m_2+...+m_t=\frac{n(n-1)}{2}$. \item There is a decomposition
${G_1,G_2,...,G_t,I}$ of $K_n$ in which $G_i$ is an $m_i$-cycle
for $i=1,2,...,t$ and $I$ is a $1$-factor if and only if $n$ is
even, $3\leq m_i\leq n$ for $i=1,2,...,t$ and
$m_1+m_2+...+m_t=\frac{n(n-2)}{2}$.
\end{enumerate}
\end{Theorem}

\section{Necessary Conditions}

As previous mentioned, the existence of a $K_3$-decomposition of ${\lambda_1} K_m
\vee_{\lambda_2} {{\lambda_1}K_n}$ provides some necessary conditions for the existence of the design \GDD$(m,n;\lambda_1,\lambda_2)$. In particular, we have three necessary conditions for the existence of \GDD$(m,n;3,\lambda)$ in the following theorem. Note that a $K_3$ or $3$-cycle in a graph is called a {\em triangle}.

\begin{Theorem}\label{NC}
Let $(m,n,\lambda)\in\N^3$, if \GDD$(m,n;3,\lambda)$ exists then $m,n$ and $\lambda$ must satisfy the following:
\begin{description}[noitemsep]
\item[     (NC1)]  $3\mid \lambda mn$,
\item[     (NC2)]  $2\mid n-1+\lambda m$ and $2\mid m-1+\lambda n$, and
\item[     (NC3)]  $\frac{\lambda}{3}\leq \frac{m-1}{n}+\frac{n-1}{m}$.
\end{description}
\end{Theorem}

\begin{proof} Since there exists a $K_3$-decomposition of the graph $3K_m
\vee_{\lambda} {3K_n}$, the number of total edges in the graph must be
divisible by $3$ which yields (NC1). Besides, every vertex in the graph with a $K_3$-decomposition must have even degree. Now a vertex in group $M$, and $N$ is of degree $3(m-1)+\lambda n$, and $3(n-1)+\lambda m$, respectively. Thus (NC2) follows.

Since the number of triangles in the entire decomposition must be greater than the number of
triangles which contain the vertices from both groups. There are $\frac{1}{6} (3m(m-1)+3n(n-1)+2\lambda mn)$ triangles in the entire decomposition and
there are $\frac{\lambda mn}{2}$ triangles that contain vertices from both groups. Hence $\frac{1}{6} (3m(m-1)+3n(n-1)+2\lambda mn) \geq \frac{\lambda mn}{2}$
which is equivalent to (NC3).
\end{proof}

Our goal for the rest of the paper is to consider whether these three necessary conditions, (NC1) to (NC3), for the existence of a \GDD$(m,n;3,\lambda)$ are sufficient.
First, we let $\S$ be the set of all ordered
triples $(m,n,\lambda)$ that satisfies only two necessary
conditions (NC1) and (NC2) which are $3\mid \lambda mn$, $2\mid n-1+\lambda m$ and $2\mid m-1+\lambda n$ .\\

Note that the necessary condition (NC2) which consists of $2\mid n-1+\lambda m$ and $2\mid m-1+\lambda n$ implies the following two statements:
\begin{enumerate}[noitemsep]
\item  $\lambda$ is even if and only if both $m$ and $n$ are odd.
\item $\lambda$ is odd if and only if $m\not\equiv n(\mod2)$.
\end{enumerate}

Together with the necessary condition (NC1), $S$ can be described explicitly in the form of a table as follows.
{\small
\begin{center}
\begin{tabular}{|c||c|c|c|c|}
\hline
   & $n\equiv 0\pmod6$ & $n\equiv1,5\pmod6$ & $n\equiv2,4\pmod6$ & $n\equiv3\pmod6$  \\ \hline \hline
  $m\equiv0\pmod6$ &                       &  $\lambda$ is odd &                       & $\lambda$ is odd  \\\hline
  $m\equiv1,5\pmod6$ & $\lambda$ is odd & $\lambda\equiv 0\pmod6$     & $\lambda\equiv 3\pmod6$     & $\lambda$ is even \\\hline
  $m\equiv2,4\pmod6$ &                       & $\lambda\equiv 3\pmod6$     &                       & $\lambda$ is odd \\\hline
  $m\equiv3\pmod6$ & $\lambda$ is odd & $\lambda$ is even & $\lambda$ is odd & $\lambda$ is even  \\\hline

\end{tabular}
\end{center}
}

We will try to give a construction of a \GDD$(m,n;3,\lambda)$ for each $(m,n,\lambda)\in \S$ in the next section, and successfully prove the sufficient conditions for most of the cases.

\section{Sufficient Conditions}

We will rely on the notion of a certain graph decomposition in our construction. Recall that {\it Hamiltonian cycle} is a cycle that visits every vertex of a graph exactly once. The next two lemmas provide the conditions of the existence of a decomposition of $K_v$ into a collection of $k$ Hamiltonian cycles and a collections of triangles for given $v$ and $k$. These results are consequences of Theorem \ref{HC} and will be the crucial tools for our main constructions.

\begin{Lemma}\label{oddKv}
Let $v, k \in \N$ be such that $v$ is odd and $1\leq k\leq \frac{v-1}{2}$. The necessary and sufficient condition for the existence of a decomposition of $K_v$ into $k$ Hamiltonian cycles and a collection of triangles
are as following table:

\begin{center}
\begin{tabular}{|c|c|c|c|c|c|c|}
\hline
  $v$ & $k$  \\ \hline\hline
  \textnormal{$v\equiv 1 (\mod 6)$} & \textnormal{ $k\equiv 0 (\mod 3)$}   \\\hline
  \textnormal{$v\equiv 3 (\mod 6)$} & \textnormal{all $k$}  \\\hline
  \textnormal{$v\equiv 5 (\mod 6)$} & \textnormal{ $k\equiv 2 (\mod 3)$}   \\\hline
\end{tabular}
\end{center}
\end{Lemma}

\begin{proof}
Assume that $K_v$ can be decomposed into $k$ Hamiltonian cycles and a collection of triangles. By Theorem \ref{HC}, since $K_v$ has $\frac{v(v-1)}{2}$ edges and $k$ Hamiltonian cycles contains $kv$ edges, the remaining edges must be divisible by 3. Therefore, we have $3\mid (\frac{v(v-1)}{2}-kv)$. The statement follows from this condition.
\end{proof}

\begin{Lemma}\label{evenKv}
Let $v, h \in \N$ be such that $v$ is even and $1\leq h\leq \frac{v-2}{2}$. The necessary and sufficient condition for the existence of a decomposition of $K_v$ into $h$ Hamiltonian cycles, one $1$-factor and a collection of triangles are as following table:

\begin{center}
\begin{tabular}{|c|c|c|c|c|c|c|}
\hline
  $v$ & $h$  \\ \hline\hline
  \textnormal{$v\equiv 0 (\mod 6)$} & \textnormal{all $h$}   \\\hline
  \textnormal{$v\equiv 2 (\mod 6)$} & \textnormal{ $h\equiv 0,3\pmod 6$}  \\\hline
  \textnormal{$v\equiv 4 (\mod 6)$} & \textnormal{ $h\equiv 1,4 \pmod 6$}   \\\hline
\end{tabular}
\end{center}
\end{Lemma}

\begin{proof}
Assume that $K_v$ can be decomposed into $h$ Hamiltonian cycles, one $1$-factor, and a collection of triangles. By Theorem \ref{HC}, since $K_v$ has $\frac{v(v-1)}{2}$ edges, and $h$ Hamiltonian cycles and one $1$-factor contains a total of $hv+\frac{v}{2}$ edges, the remaining edges must be divisible by 3. Therefore, we have $3\mid (\frac{v(v-1)}{2}-hv-\frac{v}{2})$. The statement follows from this condition.
\end{proof}

\noindent From now on, the following notations will be used for our constructions.
\begin{enumerate}

\item Since we will be dealing with multi-sets (where each element is allowed to appear more than once), $``\cup"$ in our construction will mean that the union of multi-sets.

\item For a set $S$ with $v$ elements, we use the notation $K_v(S)$ for a complete graph of order $v$
with vertex set $S$.

\item Let $G$ be a graph and $p$ be a vertex not in $G$. By $p*G$ we mean the set of triples $\{\{p,v_1,v_2\}\mid v_1v_2\in E(G)\}.$
\end{enumerate}

Remark that for each $v\in V(G)$, $p$ and $v$ will be together in $d_G(v)$ triples in $p*G$. For each pair of
$v_1,v_2\in V(G)$, $v_1,v_2$ will be together in $t$ triples in $p*G$ where $t$ is the number of edges
between $v_1$ and $v_2$ in $G$.\\

\noindent\textbf{The star construction between a graph $G$ and a set $A$}. Let $A=\{a_1,a_2,\ldots,a_{|A|}\}$.
 \begin{enumerate}
 \item[(I)] Suppose that $G$ has a decomposition into a collection of triangles $\T_1$ and $k=|A|\lambda$ Hamiltonian cycles $H_{i,j}$, for $i\in\{1,2,\ldots,|A|\}$, $j\in\{1,2,\ldots,\lambda\}$. The star construction between $G$ and $A$ provides a collection of blocks $\T_1\cup\B$ where $\B$ is the union of the set $a_i*H_{i,j}$ for $i\in\{1,2,$ $\ldots,$ $|A|\}$, $j\in\{1,2,\ldots,\lambda\}$. Then each element in $A$ and each vertex in $G$ are together in exactly $2\lambda$ blocks in $\T_1\cup\B$.

 \item[(II)] Suppose that $G$ has a decomposition into a collection of triangles $\T_2$, $|A|$ 1-factors $F_i$ for $i\in\{1,2,\ldots,|A|\}$ and $k=|A|\lambda$ Hamiltonian cycles $H_{i,j}$ for $i\in\{1,2,\ldots,|A|\}$, $j\in\{1,2,\ldots,\lambda\}$. The star construction between $G$ and $A$ provides a collection of blocks  $\T_2\cup\B_1\cup\B_2$ where $\B_1$ is the union of the set $a_i*H_{i,j}$ for $i\in\{1,2,\ldots,|A|\}$, $j\in\{1,2,\ldots,\lambda\}$ and $\B_2$ is the union of the set $a_i*F_i$ for $i\in\{1,2,\ldots,|A|\}$. Then each element in $A$ and each vertex in $G$ are together in exactly $2\lambda+1$ blocks in $\T_2\cup\B_1\cup\B_2$.
\end{enumerate}

For the rest of the paper, $M$ and $N$ will denote the disjoint sets such that $M$ has $m$ elements and $N$ has $n$ elements, and we always assume that $m>n$. Since we are interested only in $\GDD$ with two groups, we will repeatedly use $M$ and $N$ as our groups. Since we consider only \GDD with $\lambda_1=3$ and $\lambda_1<\lambda_2$, we always have $\lambda_2\geq 4$. Given any pair of $m$ and $n$, we denote the maximum value of $\lambda$ satisfying (NC1)-(NC3) by $\lambda_{max}(m,n)$, more explicitly,

$$\lambda_{max}(m,n)=\max\{\lambda : (m,n,\lambda)\in \S, \mbox{ and } \frac{\lambda}{3}\leq \frac{m-1}{n}+\frac{n-1}{m}\}.$$

The next theorem (see more details in any design theory book e.g. \cite{LR}),
 guarantees that we can decompose $3K_v$ into a collection of triangles whenever $v$ is odd.
\begin{Theorem}\textnormal{\cite{LR}}\label{3Kn}
There exists a $K_3$-decomposition of $3K_v$ if any only if $v$ is an odd integer.
\end{Theorem}

We are now in the position to construct our designs \GDD$(m,n;3,\lambda)$, when $m>n$. There are two variations of constructions; for odd $n$ (in Theorem \ref{nodd}), and for even $n$ (in Theorems \ref{oddeven1} and \ref{oddeven2}). First, we start with an odd $n$. Observe that Lemmas \ref{lemmaoddodd} and \ref{lemmaevenodd} are for odd $m$ and even $m$, respectively.

\begin{Lemma}\label{lemmaoddodd}
Let $\lambda\geq4$ and $m>n$ be positive integers such that $m$ and $n$ are odd, and $(m,n,\lambda)\in \S$. If $\lambda \leq \lfloor3(\frac{m-1}{n})\rfloor$, then there exists a \GDD$(m,n;3,\lambda)$.

\end{Lemma}

\begin{proof}

It suffices to decompose $3K_m(M)$ into a collection of triangles and $k_1=\frac{n\lambda}{2}$ Hamiltonian cycles, $\frac{\lambda}{2}$ of which for each $n\in N$. Then the star construction (I) between $3K_m(M)$ and the set $N$, together with a $K_3$-decomposition of $3K_n(N)$ as $n$ is odd (exists by Theorem \ref{3Kn}), provides a \GDD$(m,n;3,\lambda)$.

 When both $m$ and $n$ are odd, $\lambda$ is always even, and hence $\frac{n\lambda}{2}$ is an integer. Moreover, when $m\equiv 1$ or $5\pmod6$; if $n\equiv 3\pmod 6$, $3|k_1$. Otherwise, $\lambda\equiv 0\pmod 6$, and hence $k_1=\frac{n\lambda}{2}$ is also divisible by $3$. Since $\lambda \leq \lfloor3(\frac{m-1}{n})\rfloor$, we have  $k_1=\frac{n\lambda}{2}\leq 3(\frac{m-1}{2})$ which means that there are enough Hamiltonian cycles available in $3K_m(M)$. However, it is needed to consider whether it is possible to decompose each of three copies of $K_m(M)$ into our desired decomposition.

If $m\equiv 3\pmod6$, the decomposition is given immediately by Lemma \ref{oddKv}.
If $m\equiv 5\pmod6$, by Lemma \ref{oddKv}, we need to split $k_1$ into $a_1,a_2,$ and $a_3$ for each $K_m(M)$ where $k_1=a_1+a_2+a_3$ and $a_i\equiv 2\pmod 3$, which can be done because $3|k_1$ and $\frac{m-1}{2}\equiv 2 \pmod 3$. Thus the decomposition is valid for all $k_1$.

If $m\equiv 1\pmod6$, then $\frac{m-1}{2}\equiv 0 \pmod 3$. Similarly to the previous case, the decomposition can be done.
\end{proof}

\begin{Lemma}\label{lemmaevenodd}
Let $\lambda\geq4$ and $m>n$ be positive integers such that $m$ is even and $n$ is odd. Let $(m,n,\lambda)\in \S$. If $\lambda \leq \lfloor3(\frac{m-1}{n})\rfloor$, then there exists a \GDD$(m,n;3,\lambda)$.
\end{Lemma}

\begin{proof}

 It suffices to decompose $3K_m(M)$ into $k_1=\frac{n\lambda-3}{2}$ Hamiltonian cycles, three 1-factors and a collection of triangles. Since $n$ is odd, $3K_n(N)$ can be decomposed into a collection of triangles $\T$. Use $\frac{n-3}{2}$ Hamiltonian cycles to produce $n-3$ 1-factors, then there are a total of $n$ 1-factors, one of which for each element in $N$. The remaining $n(\frac{\lambda-1}{2})$ Hamiltonian cycles are for $n$ elements in $N$, and hence $\frac{\lambda-1}{2}$ for each element. Apply the star construction (II) between $3K_m(M)$ and $N$ accordingly, together with $\T$, provides our desired \GDD$(m,n;3,\lambda)$.

 When $m$ is even and $n$ is odd, $\lambda$ is always odd. Thus $\frac{n\lambda-3}{2}$ is an integer. Furthermore, for the case $m\equiv 2$ or $4\pmod6$, if $n\equiv 3\pmod 6$ then $3|k_1$. If $n\equiv 3\pmod 6$, we have $\lambda\equiv 3\pmod 6$; and hence $3|k_1$. Since $\lambda \leq \lfloor3(\frac{m-1}{n})\rfloor$, we have  $k_1=\frac{n\lambda-3}{2}\leq 3(\frac{m-2}{2})$ which means that there are enough Hamiltonian cycles available in  $3K_m(M)$.

If $m\equiv 0\pmod6$, the decomposition is given immediately by Lemma \ref{evenKv}. If $m\equiv 4\pmod6$, $\frac{m-2}{2}\equiv 1 \pmod 3$, by Lemma \ref{evenKv},  we need to split $k_1$ into $a_1,a_2,$ and $a_3$ for each $K_m(M)$ where $k_1=a_1+a_2+a_3$ and $a_i\equiv 1\pmod 3$, which can be done because of $3|k_1$.
Thus the decomposition is valid for all $k_1$.

If $m\equiv 2\pmod6$, then $\frac{m-2}{2}\equiv 0 \pmod 3$. Similarly to the previous case, the decomposition can be done.
\end{proof}

For fixed $(m,n)$ where $n$ is odd, if  $\lambda_{max}(m,n) \leq \lfloor3(\frac{m-1}{n})\rfloor$, Lemmas \ref{lemmaoddodd}-\ref{lemmaevenodd} assure that we can construct \GDD$(m,n;3,\lambda)$ for all $(m,n,\lambda)$ satisfying (NC1)-(NC3). Otherwise,  $\lambda_{max}(m,n)$ lies between $\lfloor3(\frac{m-1}{n})\rfloor$ and $3(\frac{m-1}{n})+3(\frac{n-1}{m})$, possibly along with other large $\lambda$'s, that could cause problem in our construction. However, Theorem \ref{nodd} below shows that only two values of $\lambda$, namely $\lambda_{max}(m,n)$ and $\lambda_{max}(m,n)-2$, that we need to scrutinize.
\begin{Theorem}\label{nodd}
Let $\lambda\geq4$ and $m>n$ be positive integers such that $n$ is odd. If $\lambda<\lambda_{max}(m,n)$, then there exists a \GDD$(m,n;3,\lambda)$ except for possibly  $\lambda=\lambda_{max}(m,n)-2$ when $m\equiv 0,3\pmod 6$ or $n \equiv 3\pmod 6$, and $1\leq 3(\frac{n-1}{m})<2$.

\end{Theorem}

\begin{proof}

If $\lambda \leq \lfloor3(\frac{m-1}{n})\rfloor$, by Lemmas \ref{lemmaoddodd}-\ref{lemmaevenodd}, there exists a \GDD$(m,n;3,\lambda)$. Let \[\Gamma(m,n)=\{\lambda :  (m,n,\lambda)\in \S \mbox{ and }  \lfloor3(\frac{m-1}{n})\rfloor < \lambda\leq 3(\frac{m-1}{n})+3(\frac{n-1}{m})\}.\]
Note that $\Gamma(m,n)$ is a subset of the set that contains all integers consecutively lying between $\lfloor3(\frac{m-1}{n})\rfloor$ and $3(\frac{m-1}{n})+3(\frac{n-1}{m})$. If $|\Gamma(m,n)|\leq 1$, then only $\lambda_{max}(m,n)$ could possibly be in $\Gamma(m,n)$. If $|\Gamma(m,n)|\leq 2$, then only $\lambda_{max}(m,n)$ and $\lambda_{max}(m,n)-2$ could possibly be in $\Gamma(m,n)$.

Since $m>n$, $3(\frac{n-1}{m})<3$. Thus $\Delta=(3(\frac{m-1}{n})+3(\frac{n-1}{m}))-\lfloor3(\frac{m-1}{n})\rfloor = (3(\frac{m-1}{n})-\lfloor3(\frac{m-1}{n})\rfloor)+3(\frac{n-1}{m})<1+3=4$.

{\em Case $1$ $n \not\equiv 3\pmod 6$ and $m \not\equiv 0,3\pmod 6$.} If $n \not\equiv 3\pmod 6$ and $m \not\equiv 3\pmod 6$ then $\lambda\equiv 0\pmod 6$. If $n \not\equiv 3\pmod 6$ and $m \not\equiv 0\pmod 6$ then $\lambda\equiv 3\pmod 6$. Hence, in this case we have  $|\Gamma(m,n)|\leq 1$.\\

{\em Case $2$ $n \equiv 3\pmod 6$ or $m \equiv 0,3\pmod 6$.} Since all $\lambda$'s for a fixed $(m,n)$ are either even or odd and $\Delta<4$, $|\Gamma(m,n)|\leq 2$. However, if $3(\frac{n-1}{m})<1$ then $\Delta<2$, and hence $|\Gamma(m,n)|\leq 1$. Next consider the case when $3(\frac{n-1}{m})\geq 2$. We introduce a new  construction which uses not only the star construction between $3K_m(M)$ and $N$, but also use the star construction between $3K_n(N)$ and $M$. For the latter star construction, we decompose $3K_n(N)$ into $k_2=m$ Hamiltonian cycles and a collection of triangles. This decomposition of $3K_n(N)$ contributes 2 more to the second index of the design. It yields that one element $\lambda_0 \in \Gamma(m,n)$ admits the existence of a \GDD$(m,n;3,\lambda_0)$. Therefore, in this case, there is at most one element in $\Gamma(m,n)$ with an unknown construction. Below is the argument showing that the desired decomposition is possible.

{\em Case $2.1$ $n \equiv 3\pmod 6$ or $m \equiv 3\pmod 6$.} It can be done by the same way as the decomposition in Lemma \ref{lemmaoddodd} using $\lambda=2$. This works because of  $3(\frac{n-1}{m})\geq 2$, so the assumption in Lemma \ref{lemmaoddodd} is satisfied.

{\em Case $2.2$ $n \equiv 3\pmod 6$ or $m \equiv 0\pmod 6$.} It can be done by the similar way as the decomposition in Lemma \ref{lemmaoddodd} using $\lambda=2$. If $n \equiv 3\pmod 6$, by Lemma \ref{oddKv} the desired decomposition is obviously possible. If $n \equiv 1 \mbox{ or }5\pmod 6$, then $m \equiv 0\pmod 6$. Thus, $3|k_2=m$. Hence, we can split $k_2$ into suitable $a_1$, $a_2$ and $a_3$ for each $K_n(N)$ where $k_2=a_1+a_2+a_3$ and $a_i\equiv 0$ or $2\pmod 3$ depending on the value of $n$ as required.
\end{proof}


For the case that $n$ is even (so, $m$ is odd) when $m>n$, the same construction as previous case no longer works because there is no decomposition of $3K_n$ into triangles if $n$ is even as in Lemma \ref{3Kn}. A modified construction will be presented here. The construction is separated into two cases depending on $m\equiv 1 \mbox{ or } 5\pmod 6$ and $m\equiv 3\pmod 6$ showed in Theorems \ref{oddeven1} and \ref{oddeven2}, respectively.

\begin{Lemma}\label{lemmaoddeven1}
Let $\lambda\geq4$ and $m>n$ be positive integers such that $m\equiv 1 \mbox{ or } 5\pmod 6$ and $n$ is even. Let $(m,n,\lambda)\in \S$. If $\lambda \leq \lfloor3(\frac{m-3}{n})\rfloor$ and $\lambda\leq 3(n-1)$, then there exists a \GDD$(m,n;3,\lambda)$.

\end{Lemma}

\begin{proof}

 Fix a vertex $a$ in $M$. Since $\lambda\leq 3(n-1)$, there exists a \GDD$(n,1;3,\lambda)$ on $N$ and $\{a\}$ by Lemma \ref{lemmaevenodd}. We will employ Lemma \ref{evenKv} to decompose $3K_{m-1}(M\setminus\{a\})$ into $k_1=\frac{n\lambda}{2}$ Hamiltonian cycles, three 1-factors and a collection of triangles. Then we will use a modified star construction as follows: The three 1-factors are for the vertex $a$. Among $k_1$ Hamiltonian cycles, we use $\frac{\lambda-1}{2}$ Hamiltonian cycles for each vertex in $N$, and the remaining $\frac{n}{2}$ Hamiltonian cycles to produce $n$ 1-factors, one of which for each vertex in $N$. Then these all will yield a \GDD$(m,n;3,\lambda)$ on $M$ and $N$.

 Now it suffices to show that the desired decomposition of $3K_{m-1}(M\setminus\{a\})$ is possible. Since $n$ is even, $\frac{n\lambda}{2}$ is an integer. Since $\lambda \leq \lfloor3(\frac{m-3}{n})\rfloor < 3(\frac{m-3}{n})$, we have  $k_1=\frac{n\lambda}{2}\leq 3(\frac{m-3}{2})$ which means that there are enough Hamiltonian cycles available in  $3K_{m-1}(M\setminus\{a\})$. If $m\equiv 1\pmod6$, then $m-1\equiv 0\pmod6$; so, the decomposition is obviously possible by Lemma \ref{evenKv}. When $m\equiv 5\pmod 6$, if $n\equiv 0\pmod6$ then $\lambda$ is odd; if $n\equiv 2$ or $4\pmod6$ then $\lambda\equiv3\pmod6$. It follows that we always have $3|k_1$. Now assume that the $i^{th}$ copy of $K_{m-1}(M\setminus\{a\})$ provide $a_i$ Hamiltonian cycles, so $k_1=a_1+a_2+a_3$.  By Lemma \ref{evenKv}, $a_i\equiv 1\pmod3$. Furthermore, the maximum number of Hamiltonian cycles in $K_{m-1}(M\setminus\{a\})$ is $\frac{m-3}{2}$ which is also congruence to $1\pmod3$. Thus it is possible to find a proper $a_i$ for $i=1,2,3$, and therefore, the decomposition can be done.
 \end{proof}

\begin{Lemma}\label{lemma2oddeven1}
Let $\lambda\geq4$ and $m>n$ be positive  integers such that $m\equiv 1 \mbox{ or } 5\pmod 6$ and $n$ is even. Let $(m,n,\lambda)\in \S$ be such that $\lambda< \lambda_{max}(m,n)$. If $\lambda\leq 3(n-1)$, then there exists a \GDD$(m,n;3,\lambda)$ except for possibly  $\lambda=\lambda_{max}(m,n)-2$ when $n \equiv 0\pmod 6$.

\end{Lemma}

\begin{proof}

 The proof is carried out similarly to Theorem \ref{nodd}. By Lemma \ref{lemmaoddeven1}, it remains to consider the case $\lfloor3(\frac{m-3}{n})\rfloor < \lambda\leq 3(\frac{m-1}{n})+3(\frac{n-1}{m})$.  Note that $\Delta=(3(\frac{m-1}{n})+3(\frac{n-1}{m}))-\lfloor3(\frac{m-3}{n})\rfloor \leq 3(\frac{2}{n})+3(\frac{n-1}{m})<\frac{6}{n}+3$.  Since $4\leq\lambda\leq 3(n-1)$, $n\neq 2$; so $n\geq 4$. Then $\Delta<\frac{6}{n}+3 <6$. If $n\not\equiv 0\pmod6$, then we have $\lambda\equiv 3\pmod 6$. Hence this implies that only $\lambda=\lambda_{max}(m,n)$ that may or may not fit our construction scheme. On the other hand, $n\equiv 0\pmod6$; so $n\geq 6$ and $\lambda$ is odd. It follows that $\Delta<4$. Thus it makes  $\lambda_{max}(m,n)$ and $\lambda_{max}(m,n)-2$ the only possible unsolve cases.
\end{proof}

\begin{Theorem}\label{oddeven1}
Let $\lambda\geq4$ and $m>n$ be positive integers such that $m\equiv 1 \mbox{ or } 5\pmod 6$ and $n$ is even. Let $(m,n,\lambda)\in \S$ be such that $\lambda<\lambda_{max}(m,n)$.
\begin{itemize}[noitemsep,topsep=0pt]
\item[$(i)$] If $n\geq \sqrt{m}+1$, then there exists a \GDD$(m,n;3,\lambda)$ except for possibly  $\lambda=\lambda_{max}(m,n)-2$ when $n \equiv 0\pmod 6$.
\item[$(ii)$] If $n\leq \sqrt{m}$ and $\lambda\leq3(n-1)$, then there exists a \GDD$(m,n;3,\lambda)$.
\end{itemize}
\end{Theorem}

\begin{proof}
$(i)$ If $n\geq \sqrt{m}+1$, then $\lambda\leq 3(\frac{m-1}{n})+3(\frac{n-1}{m})\leq 3(\frac{(n-1)^2-1}{n})+3(\frac{n-1}{m})=3(n-2+\frac{n-1}{m})\leq 3(n-1)$. Therefore the statement holds by Lemma \ref{lemma2oddeven1}.

$(ii)$ By Lemma \ref{lemma2oddeven1}, we need to consider only when $n\equiv 0\pmod6$; so, $n\geq6$. We have $m\geq 9$ since $n\leq \sqrt{m}$, thus $\sqrt{m}\leq\frac{m}{3}$. Hence $n\leq \sqrt{m}\leq\frac{m}{3}$, which yields $3(\frac{n-1}{m})<1$.  So, the upper bound of $\Delta$ in Lemma \ref{lemma2oddeven1} becomes $\Delta=(3(\frac{m-1}{n})+3(\frac{n-1}{m}))-\lfloor3(\frac{m-3}{n})\rfloor \leq 3(\frac{2}{n})+3(\frac{n-1}{n})<\frac{6}{n}+1\leq2$. This concludes that the construction works for all $\lambda\neq\lambda_{max}(m,n)$.
\end{proof}

Remark that our construction in Theorem \ref{oddeven1} does not include the case when $n\leq \sqrt{m}$ and $\lambda>3(n-1)$. It is left as an open problem.


For the last case $m\equiv 3\pmod 6$ and even $n$, the proof is carried out similarly to the case odd $m\equiv 1,5\pmod 6$ and even $n$. However, we need to pull  three vertices out of $M$ instead of one vertex, which needs a stronger assumption. Consequently, our construction in this case cannot account for $\lambda=\lambda_{max}(m,n)$ and $\lambda=\lambda_{max}(m,n)-2$.

\begin{Lemma}\label{lemmaoddeven2}
Let $\lambda\geq4$ and $m>n$ be positive integers such that $m\equiv 3\pmod 6$ and $n$ is even. Let $(m,n,\lambda)\in \S$. If $\lambda \leq \lfloor3(\frac{m-7}{n})\rfloor$ and $\lambda\leq n-1$, then there exists a \GDD$(m,n;3,\lambda)$.

\end{Lemma}

\begin{proof}

 Fix three vertices $a,b$ and $c$ in $M$. Since $\lambda\leq n-1$, there exists a \GDD$(n,3;3,\lambda)$ on $N$ and $\{a,b,c\}$ by Lemma \ref{lemmaevenodd}.  Now $m-3\equiv 0\pmod 6$. We will employ Lemma \ref{evenKv} to decompose $3K_{m-3}(M\setminus\{a,b,c\})$ into $k_1=\frac{n\lambda}{2}+3$ Hamiltonian cycles, three 1-factors and a collection of triangles. Similar to the proof of Lemma \ref{lemmaoddeven1}, the bound $\lambda \leq \lfloor3(\frac{m-7}{n})\rfloor$ provides a proper construction.
 \end{proof}

\begin{Lemma}\label{lemma2oddeven2}
Let $\lambda\geq4$ and $m>n$ be positive integers such that $m\equiv 3\pmod 6$ and $n$ is even. Let $(m,n,\lambda)\in \S$ be such that $\lambda< \lambda_{max}(m,n)-2$. If $\lambda\leq n-1$, then there exists a \GDD$(m,n;3,\lambda)$ except for possibly  $\lambda=\lambda_{max}(m,n)-4$ when $6\leq n\leq 16$.

\end{Lemma}

\begin{proof}

 The proof is carried out similarly to Theorem \ref{nodd}. By Lemma \ref{lemmaoddeven2}, it remains to consider the case $\lfloor3(\frac{m-7}{n})\rfloor < \lambda\leq 3(\frac{m-1}{n})+3(\frac{n-1}{m})$.  We have $\Delta=(3(\frac{m-1}{n})+3(\frac{n-1}{m}))-\lfloor3(\frac{m-7}{n})\rfloor \leq 3(\frac{6}{n})+3(\frac{n-1}{m})<\frac{18}{n}+3$. Note that $\lambda$ is always odd. If $n\geq 18$, then $\Delta<4$. Thus at most two largest values of $\lambda$ that do not fit in our construction scheme. Furthermore, since $4\leq\lambda\leq n-1$, $n\geq5$. Then it remains to consider $6\leq n\leq16$, which we have $\Delta<6$. This leaves possibly at most three values of $\lambda$ that we cannot guarantee the existence of such $\GDD$.
\end{proof}

\begin{Theorem}\label{oddeven2}
Let $\lambda\geq4$ and $m>n$ be positive integers such that $m\equiv 3\pmod 6$ and $n$ is even. Let $(m,n,\lambda)\in \S$ be such that $\lambda<\lambda_{max}(m,n)-2$.
\begin{itemize}[noitemsep,topsep=0pt]
\item[$(i)$] If $n\geq \sqrt{3m}+2$, then there exists a \GDD$(m,n;3,\lambda)$ except for possibly  $\lambda=\lambda_{max}(m,n)-4$ when $6\leq n\leq 16$.
\item[$(ii)$] If $n\leq \sqrt{3m}+1$ and $\lambda\leq n-1$, then there exists a \GDD$(m,n;3,\lambda)$ except for possibly  $\lambda=\lambda_{max}(m,n)-4$ when $(m,n)\in\{(21,6),(27,6)\}$ .
\end{itemize}

\end{Theorem}

\begin{proof}
$(i)$ If $n\geq \sqrt{3m}+2$, then $\lambda\leq 3(\frac{m-1}{n})+3(\frac{n-1}{m})\leq \frac{3}{n}(\frac{(n-2)^2}{3}-1)+3(\frac{n-1}{m})=(n-4+\frac{1}{n})+3(\frac{n-1}{m})\leq n-1+\frac{1}{n}$. But $\lambda$ is an integer, so $\lambda\leq n-1$. Therefore the statement holds by Lemma \ref{lemma2oddeven2}.

$(ii)$ Since $n\leq \sqrt{3m}+1$, we have $3(\frac{n-1}{m})\leq\sqrt{\frac{27}{m}}$. The upper bound of $\Delta$ in Lemma \ref{lemma2oddeven2} becomes $\Delta=(3(\frac{m-1}{n})+3(\frac{n-1}{m}))-\lfloor3(\frac{m-7}{n})\rfloor \leq 3(\frac{6}{n})+3(\frac{n-1}{m})<\frac{18}{n}+\sqrt{\frac{27}{m}}$.  If $n\geq 10$ and $m\geq 9$, then $\Delta <\frac{18}{n}+\frac{3\sqrt{3}}{m}<2+2=4$. Therefore, the construction works for all $\lambda\neq\lambda_{max}(m,n)$ and $\lambda\neq \lambda_{max}(m,n)-2$. Consider the remaining small $m$ and $n$. Since $4\leq\lambda\leq n-1$, $n\geq 5$; and so $m>3$. Thus $\Delta<6$. Together with the assumptions $n\leq \sqrt{3m}+1$ and $\lambda\leq 3(\frac{m-1}{n})+3(\frac{n-1}{m})$, we are able to conclude that only $(m,n)\in\{(21,6),(27,6)\}$ may or may not fit our construction scheme for \GDD$(m,n;3,\lambda)$ where $\lambda=\lambda_{max}(m,n)-4$.
\end{proof}

 It is noted that our construction in Theorem \ref{oddeven2} does not include the case when $n\leq \sqrt{3m}+1$ and $\lambda> n-1$. The existence of $\GDD$  in this case therefore remains open.
\section{Conclusion and Open Problem}

Our constructions assure that the necessary conditions of the existence of our $\GDD$ in Theorem \ref{NC} are sufficient for most of the cases. Given $m$ and $n$ such that $m>n$.  When $n$ is odd, there are at most two values of $\lambda$, namely $\lambda_{max}(m,n)$ and $\lambda_{max}(m,n)-2$, that the existence of  \GDD$(m,n;3,\lambda)$ remains open. However, when $n$ is even (so, $m$ is odd), there are at most three values of $\lambda$, namely $\lambda_{max}(m,n)$, $\lambda_{max}(m,n)-2$ and $\lambda_{max}(m,n)-4$,  that the problem remains unsolved provided that $(m,n,\lambda)$ does not satisfy one of the following:
\begin{itemize}[noitemsep,topsep=0pt]
\item[ $(i)$]  $m\equiv 1 \mbox{ or } 5\pmod 6$, $n\leq\sqrt{m}$ and $\lambda>3(n-1)$.
\item[ $(ii)$] $m\equiv 3\pmod 6$, $n\leq\sqrt{3m}+1$ and $\lambda>n-1$.\\
\end{itemize}

\noindent Theorems \ref{Main1}-\ref{Main3} conclude our results along with some open problems. Recall the necessary conditions for the existence of \GDD$(m,n;3,\lambda)$ in Theorem \ref{NC} are as follows.
\begin{description}[noitemsep]
\item[     (NC1)]  $3\mid \lambda mn$,
\item[     (NC2)]  $2\mid n-1+\lambda m$ and $2\mid m-1+\lambda n$, and
\item[     (NC3)]  $\frac{\lambda}{3}\leq \frac{m-1}{n}+\frac{n-1}{m}$.
\end{description}

\begin{Theorem}\label{Main1}
Let $\lambda\geq4$ and $m>n$ be positive integers such that $n$ is odd. Then the necessary condition for the existence of \GDD$(m,n;3,\lambda)$ is also sufficient except for possibly
\begin{enumerate}[noitemsep,topsep=0pt]
\item $\GDD(m,n;3,\lambda_{max}(m,n))$, and
\item $\GDD(m,n;3,\lambda_{max}(m,n)-2)$ if $m \equiv 0,3\pmod 6$ or $n \equiv 3\pmod 6$ and $1\leq 3(\frac{n-1}{m})<2$.
\end{enumerate}
\end{Theorem}

\begin{Theorem}\label{Main2}
Let $\lambda\geq4$ and $m>n$ be positive integers such that $m\equiv 1 \mbox{ or } 5\pmod 6$ and $n$ is even. Then the necessary condition for the existence of \GDD$(m,n;3,\lambda)$ is also sufficient except for possibly
\begin{enumerate}[noitemsep,topsep=0pt]
\item $\GDD(m,n;3,\lambda_{max}(m,n))$,
\item $\GDD(m,n;3,\lambda_{max}(m,n)-2)$ if $n\equiv 0\pmod 6$ and $n\geq\sqrt{m}+1$, and
\item $\GDD(m,n;3,\lambda)$ if $n\leq\sqrt{m}$ and $\lambda>3(n-1)$.
\end{enumerate}
\end{Theorem}

\begin{Theorem}\label{Main3}
Let $\lambda\geq4$ and $m>n$ be positive integers such that $m\equiv 3\pmod 6$ and $n$ is even. Then the necessary condition for the existence of \GDD$(m,n;3,\lambda)$ is also sufficient except for possibly
\begin{enumerate}[noitemsep,topsep=0pt]
\item $\GDD(m,n;3,\lambda_{max}(m,n))$ and $\GDD(m,n;3,\lambda_{max}(m,n)-2)$,
\item $\GDD(m,n;3,\lambda_{max}(m,n)-4)$ if $n\geq\sqrt{3m}+2$ and $6\leq n\leq16$.
\item $\GDD(m,n;3,\lambda_{max}(m,n)-4)$ if $(m,n)\in\{(21,6),(27,6)\}$, and
\item $\GDD(m,n;3,\lambda)$ if $n\leq\sqrt{3m}+1$ and $\lambda>n-1$.

\end{enumerate}
\end{Theorem}

 Note that the question whether the constructions for the remaining cases are possible still remains unsolved.


\begin{thebibliography}{200}

\bibitem{AL} B. Alspach, Research Problem 3, \textit{Discrete Math.} 36 (1981) 333.

\bibitem{BA} P. Balister, On the Alspach conjecture, \textit{Combin. Probab. Comput.} 10 (2001), 95--125.

\bibitem{BO} R.C. Bose and T. Shimamoto, Classification and analysis of partially balanced incomplete block designs with two associate classes, \textit{J. Amer. Statist. Assoc.} 47(1952), 151--184.

\bibitem{HC} D. Bryant, D. Horsley and W. Pettersson, Cycle decompositions V: Complete graphs into cycles of arbitrary lengths, \textit{Proc. London Math. Soc.} (3), (2013) doi: 10.1112/plms/pdt051.

\bibitem{1L} A. Chaiyasena, S.P. Hurd, N. Punnim and D.G.Sarvate. Group divisible design with two association classes, \textit{J. Combin. Math. Combin. Comput.} 82(1) (2012), 179--198.

\bibitem{12G} S. I. El-Zanati, N. Punnim and C. A. Rodger, Gregarious GDDs with Two Associate Classes, \textit{Graphs Combin.} 26(6) (2010),775--780.

\bibitem{D1} W. Lapchinda and N. Pabhapote, Group divisible designs with two associate classes and $\lambda_1-\lambda_2=1$ , \textit{Int. J. Pure Appl. Math.} 54, No.4 (2009), 601--608.

\bibitem{LR} C. C. Lindner and C. A. Rodger, \textit{Design Theory} (CRC Press, Boca Raton, 1997).

\bibitem{nittiya} N. Pabhapote, Group divisible designs with two associate classes and with two uniqual groups, \textit{Int. J. Pure Appl. Math.} \textbf{81(1)} (2012), 191--198.

\bibitem{L23} N. Pabhapote and A. Chaiyasena, Group Divisible Designs with Two Associate Classes and $\lambda_2=3$, \textit{Int. J. Pure Appl. Math.}  71, No.3 (2011), 455--463.

\bibitem{L21} N. Pabhapote and N. Punnim, Group Divisible Designs with Two Associate Classes and $\lambda_2=1$, \textit{Int. J. Math. Math. Sci.} (2011),  Article ID 148580, 10 pages, doi:10.1155/2011/148580.

\bibitem{12} N. Punnim and C. Uiyyasathian, Group Divisible Designs with Two Associate Classes and $(\lambda_1,\lambda_2)=(1,2)$, \textit{J. Combin. Math. Combin. Comput.} 82(1) (2012), 117--130.

\bibitem{14} P. Sinsap and C. Uiyyasathian, Group Divisible Designs \GDD$(m,n;2,\lambda)$, submitted.


\bibitem{13} C. Uiyyasathian and N. Punnim, Some construction of Group Divisible Designs \GDD$(m,n;1,3)$, \textit{Int. J. Pure Appl. Math.}
    104(1)(2015), 19--28.



\end{thebibliography}
\end{document}